\documentclass[12pt,reqno]{amsart}

\usepackage{amssymb,latexsym}

\usepackage{enumerate}
\allowdisplaybreaks
\usepackage[french,english]{babel}
\usepackage{amsmath}
\usepackage{graphicx}
\usepackage{amssymb}
\usepackage{bbm}
\usepackage{amsthm,mathtools}
\usepackage{ulem}
\usepackage{geometry}
\usepackage{tikz-cd}
\usepackage{mathrsfs}
\usepackage[colorinlistoftodos]{todonotes}
\usepackage{enumitem}
\usepackage{verbatim}
\usepackage[foot]{amsaddr}
\usepackage{dsfont}
\usepackage{cite}
\usepackage[T1]{fontenc}

\makeatletter

\@namedef{subjclassname@2010}{
	
	\textup{2020} Mathematics Subject Classification}

\makeatother
\newtheorem{thm}{Theorem}[section]
\newtheorem*{thm*}{Theorem}

\newtheorem{lem}[thm]{Lemma}

\theoremstyle{definition}

\numberwithin{equation}{section}

\newcommand{\bg}{\big}

\newcommand{\bgg}{\bigg}

\newcommand{\Bg}{\Big}

\newcommand{\inv}{^{-1}}

\newcommand{\mbn}{\mathbb{N}}

\newcommand{\mcg}{\mathcal{G}}

\newcommand{\mmd}{\mathrm{d}}
\newcommand{\mme}{\mathrm{e}}

\newcommand{\ol}{\overline}

\usepackage{hyperref}
\hypersetup{colorlinks=true,linkcolor=blue,anchorcolor=blue,citecolor=blue}
\usepackage{color}
\newcommand{\newabstract}[1]{%
	\par\bigskip
	\csname otherlanguage*\endcsname{#1}%
	\csname captions#1\endcsname
	\item[\hskip\labelsep\scshape\abstractname.]
}

\begin{document}

\title[Extreme values of derivatives of Dirichlet $L$-functions]{Extreme values of derivatives of Dirichlet $L$-functions}

\author{Xinghua Cui\textsuperscript{1}}
\author{Zhifeng Peng\textsuperscript{1}}
\author{Yutong Song\textsuperscript{2,3}}
\author{Shengbo Zhao\textsuperscript{2}}

\address{1. School of Mathematical Sciences, Soochow University, Suzhou 215006, China}
\address{2. School of Mathematical Sciences, Key Laboratory of Intelligent Computing and Applications (Tongji University), Ministry of Education, Tongji University, Shanghai 200092, China}	
\address{3. Institute of Analysis and Number Theory, Graz University of Technology, Steyrergasse 30, 8010 Graz, Austria}

\email{20234007004@stu.suda.edu.cn}
\email{zfpeng@suda.edu.cn}

\email{99yutongsong@gmail.com}
\email{shengbozhao@hotmail.com}

	\begin{abstract} 
    In this paper, we establish lower bounds for extreme values of derivatives of Dirichlet \(L\)-functions in the range \(1/2<\sigma<1\). Compared with the work of Aistleitner, Mahatab, Munsch, and Peyrot in 2019, our result shows that derivatives of Dirichlet \(L\)-functions can attain extreme values of the same order of magnitude as the original \(L\)-functions when \(1/2<\sigma<1\) holds.
	\end{abstract}
	
	
	
	\maketitle

\section{Introduction}

Dirichlet \(L\)-functions are fundamental objects in analytic number theory. Their extreme values are closely connected with several important problems, including estimates for character sums, the distribution of quadratic residues, and the generalized Riemann hypothesis. Consequently, the study of their extreme values has attracted considerable attention.
\par
The study of extreme values of Dirichlet \(L\)-functions can be traced back to the work of Littlewood. Granville and Soundararajan \cite{Granville2006Ramanujan} established a sharp lower bound for Dirichlet \(L\)-functions at \(s=1\). They showed that for any \(A \ge 10\) and any sufficiently large prime \(q\), there exist at least \(q^{1-1/A}\) characters \(\chi \pmod{q}\) such that
\[
|L(1,\chi)| \ge \mme^\gamma\bg(\log\log q+\log\log\log q-\log\log\log\log q-\log A+O(1)\bg).
\]
Here, \(\gamma\) denotes the Euler–Mascheroni constant. Subsequently, Aistleitner, Mahatab, Munsch, and Peyrot \cite{Aistleitner2019QJMath} sharpened this result by proving that, for sufficiently large primes \(q\), there exists a non-principal character \(\chi \pmod q\) such that
\[
    |L(1,\chi)| \ge \mme^\gamma\bg(\log\log q+\log\log\log q-C+o(1)\bg),
\]
where \(C=1+\log\log 4 \approx 1.33\). This is currently the best known result. Moreover, they obtained a probabilistic result concerning the characters \(\chi\) for which such extreme values occur. Moreover, for any fixed \(1/2<\sigma<1\), they proved that there exists a non-principal character \(\chi \pmod q\) such that 
\begin{align}\label{ammp}
    \log |L(\sigma,\chi)| \ge C(\sigma)\frac{(\log q)^{1-\sigma}}{(\log\log q)^{\sigma}}
\end{align}
for some constant \(C(\sigma)>0\).
\par
D. Yang \cite{yang2024BLMS} studied in depth extreme values of derivatives of Dirichlet \(L\)-functions at \(s=1\). By combining estimates for smooth numbers with the resonance method, he proved that, for sufficiently large primes \(q\), 
\[
\max_{\substack{\chi \neq \chi_0 \\ \chi \pmod q}} \bg|L^{(\ell)}(1,\chi) \bg| \ge \bg(Y_\ell+o(1)\bg) (\log\log q)^{\ell+1}
\]
uniformly holds for all positive integers \(\ell \le (\log\log\log q)/(\log\log\log\log q)\). Here, \(Y_\ell = \int_0^\infty u^\ell \rho(u)\mmd u\) and \(\rho(u)\) denotes the Dickman function. For studies on the extreme values of derivatives of Dirichlet \(L\)-functions and related functions, see \cite{dong2023BAustMS,dong2023Onde,yang2022extreme,YANG2026jmaa}.
\par
The resonance method is currently the main tool for investigating extreme values. Its basic idea originates from Voronin’s work \cite{voronin1988lower} in 1988. Soundararajan \cite{soundararajan2008extreme} subsequently improved Voronin’s method and made it more concise and efficient. Further refinements were introduced by Aistleitner \cite{aistleitner2016lower}, Bondarenko and Seip \cite{bondarenko2017Duke,bondarenko2018MathAnn}, and several other authors. For more details on the resonance method, we recommend \cite{aistleitner2019IMRN,chirre2019extreme,YANG2026jnt} and the references therein.
\par
Inspired by D. Yang's work \cite{yang2024BLMS}, we study extreme values of derivatives of Dirichlet \(L\)-functions for \(1/2<\sigma<1\). By applying the long resonator method, we obtain the following theorem:

\begin{thm}
\label{thm1}
    Fix \(\ell \in \mbn\) and \(\sigma \in (1/2,1)\). Then there exists a constant $c_{\sigma,\ell}>0$ such that, for every sufficiently large prime $q$, one can find a non-principal Dirichlet character $\chi \pmod q$ satisfying
    \[
    \bg| L^{(\ell)}(\sigma,\chi) \bg| \ge \exp \bgg(c_{\sigma,\ell} \frac{(\log q)^{1-\sigma}}{(\log\log q)^{\sigma}} \bgg).
    \]
\end{thm}

\par
In the case \(\ell=0\), Theorem \ref{thm1} coincides with \eqref{ammp}, and hence recovers the result of Aistleitner et al. \cite{Aistleitner2019QJMath}. A similar result on the extreme values of derivatives of the Riemann zeta function can be found in \cite[Theorem 2(B)]{yang2022extreme}, and the present work generalizes that result to Dirichlet \(L\)-functions.
\par
This paper is organized as follows. In Section \ref{sec-pre}, we introduce the resonator and establish a truncation formula for derivatives of Dirichlet \(L\)-functions. In Section \ref{sec-prove}, we apply the long resonator method to prove Theorem \ref{thm1}.

\section{Preliminaries}
\label{sec-pre}
In this section, we carry out some preliminary work. We begin with the following asymptotic formula, which shows that derivatives of Dirichlet \(L\)-functions can be effectively approximated by a truncated Dirichlet series.

\begin{lem}
\label{lem1}
    Fix \(\ell \in \mbn_+\) and \(\sigma \in (1/2,1)\). When \(q\) is sufficiently large, for any non-principal character \(\chi\pmod {q}\), we have
    \[
    (-1)^\ell  L^{(\ell)}(\sigma,\chi) = \sum_{n \le q}\frac{(\log n )^\ell \chi(n)}{n^\sigma} + O_{\sigma,\ell} \bg( q^{\frac{1}{2}-\sigma}(\log q)^{\ell+1}\bg).
    \]
\end{lem}
\begin{proof}
    Trivially, we have
\begin{align}
    \label{lem1-dirichlet}
    (-1)^\ell  L^{(\ell)}(\sigma,\chi) = \sum_{n\in \mbn} \frac{(\log n )^\ell \chi(n)}{n^\sigma} = \Bg(\sum_{n \le q} +\sum_{n > q}\Bg) \frac{(\log n )^\ell \chi(n)}{n^\sigma}.
\end{align}
Set
\[
A_\chi (t) = \sum_{n \le t} \chi(n) ~\text{and}~ f(t) = \frac{(\log t)^\ell}{t^\sigma}.
\]
The P\'olya-Vinogradov inequality gives, uniformly for \(t \ge 1\), 
\begin{align}
    \label{lem1-pv}
    |A_\chi (t)| \ll \sqrt{q}\log q.
\end{align}
Hence, an application of partial summation shows that the tail sum on the right-hand side of \eqref{lem1-dirichlet} satisfies
\begin{align}
    \label{lem1-tail}
    \sum_{n > q}\frac{(\log n )^\ell \chi(n)}{n^\sigma} = -A_\chi (q) f(q) + \int_{q}^\infty A_\chi(t) f^\prime (t) \mmd t. 
\end{align}
Noting that
\[
\bg|f^\prime(t)\bg| \ll_{\sigma,\ell} \frac{(\log t)^\ell}{t^{\sigma+1}}.
\]
Combining this with \eqref{lem1-pv} and \eqref{lem1-tail}, we obtain
\[
\sum_{n > q}\frac{(\log n )^\ell \chi(n)}{n^\sigma} \ll_{\sigma,\ell} q^{\frac{1}{2}-\sigma}(\log q)^{\ell+1}.
\]
Substituting this into \eqref{lem1-dirichlet} completes the proof of Lemma \ref{lem1}.
\end{proof}

\par
Next, we construct the resonator \(R(\chi)\). Define \(Y= \eta \log q \log\log q\), where \(0 < \eta < (2\log 4)\inv \sigma\). As in \cite{Aistleitner2019QJMath}, define \(r(n)\) to be a completely multiplicative function whose values at primes \(p\) are given by
	\[
	r(p)=
	\begin{cases}
		1/2, & p\leq Y,\\
		0, & p>Y.
	\end{cases}
	\]
Furthermore, define the resonator
\[
R(\chi) = \sum_{n \in \mbn}r(n)\chi(n).
\]
Since \(r(n)\) is multiplicative, \(R(\chi)\) can be written in the following Euler product form:
\[
R(\chi) = \prod_{p \le Y} \bg(1-r(p)\chi(p) \bg)\inv = \prod_{p \le Y} \Bg(1-\frac{\chi(p)}{2} \Bg)\inv.
\]

\section{Proof of Theorem \ref{thm1}}
\label{sec-prove}

In this section, we prove Theorem \ref{thm1} using the long resonator method. Let \(\mcg_q\) denote the set of all Dirichlet characters \(\chi\) modulo \(q\), and write
\[
P(\chi) : = P_\ell(\chi) = \sum_{n \le q}\frac{(\log n )^\ell \chi(n)}{n^\sigma}.
\]
Furthermore, define the following two sums:
\begin{align*}
    & S_1 := S_1(R,q) = \sum_{\chi \in \mcg_q}|R(\chi)|^2, \\
    & S_2 := S_2(R,q) = \sum_{\chi \in \mcg_q}P(\chi) |R(\chi)|^2.
\end{align*}

\par
For \(S_1\), we expand \(|R(\chi)|^2\) and apply the orthogonality of Dirichlet characters, which gives 
\begin{align}
\label{S1eq}
    S_1 = \sum_{\chi \in \mcg_q} \sum_{m,n \ge 1} r(m)r(n)\chi(m)\ol\chi(n) = \phi(q)\sum_{m \equiv n \pmod q} r(m)r(n),
\end{align}
where \(\phi(n)\) is Euler’s totient function. As \(r(1)=1\), retaining only the term \(m=n=1\) and using the classical lower bound for Euler’s totient function, we obtain the following crude lower bound for \(S_1\):
\begin{align}
    \label{S1lower}
    S_1 \ge q^{1-o(1)}.
\end{align}

\par
For the principal character \(\chi_0\), it follows from the prime number theorem that
\[
|R(\chi_0)|^2 = 2^{2\pi(Y)} = q^{\eta\log 4+o(1)}.
\]
Together with
\[
|P(\chi_0)| \ll_{\sigma,\ell} \sum_{n \le q} \frac{(\log n)^\ell}{n^\sigma} \ll_{\sigma,\ell} q^{1-\sigma}(\log q )^\ell,
\]
this yields
\[
|P(\chi_0)||R(\chi_0)|^2 \ll_{\sigma,\ell} q^{1-\sigma+\eta\log 4+o(1)}.
\]
It follows from \(0 < \eta < (2\log 4)\inv \sigma\) that 
\[
1-\sigma+\eta\log4 <1.
\]
Therefore, the contribution of the principal character \(\chi_0\) is negligible compared with the lower bound for \(S_1\) given in \eqref{S1lower}.
\par
Denote by \(\mathcal{G}_q^\ast\) the set of all non-principal Dirichlet characters modulo \(q\), that is, \(\mcg_q^\ast = \mcg_q \setminus \{\chi_0\}\). We then define the following two sums:
\begin{align*}
    & S_1^\ast := S_1^\ast(R,q) = \sum_{\chi \in \mcg_q^\ast}|R(\chi)|^2, \\
    & S_2^\ast := S_2^\ast(R,q) = \sum_{\chi \in \mcg_q^\ast}P(\chi) |R(\chi)|^2.
\end{align*}
Combining \(S_1^\ast \le S_1\) and \eqref{S1lower}, we have
\[
\frac{S_2^\ast}{S_1^\ast} \ge \frac{S_2}{S_1} + O\bg(q^{-\sigma+\eta\log 4+o(1)} \bg).
\]
Thus,
\begin{align}
    \label{max1}
    \max_{\chi\in \mcg_q^\ast} |P(\chi)| \ge \frac{S_2}{S_1} + O\bg(q^{-\sigma+\eta\log 4+o(1)} \bg).
\end{align}
\par
We now establish a lower bound for the ratio \(S_2/S_1\). To this end, expanding \(P(\chi)\) and \(|R(\chi)|^2\) in \(S_2\), we obtain that 
\begin{align*}
    S_2 &\, = \sum_{\chi \in \mcg_q}\sum_{k \le q}\frac{(\log k )^\ell \chi(k)}{k^\sigma} \sum_{m,n\ge 1} r(m)r(n)\chi(m)\ol\chi(n)  \\
    &\, \sum_{k \le q}\frac{(\log k )^\ell}{k^\sigma}\sum_{m,n\ge 1} r(m)r(n)  \sum_{\chi \in \mcg_q}\chi(k)\chi(m)\ol\chi(n).
\end{align*}
The orthogonality of Dirichlet characters implies that
\[
    S_2 = \sum_{k \le q}\frac{(\log k )^\ell}{k^\sigma} \phi(q) \sum_{\substack{m,n \ge 1 \\ mk \equiv n\pmod q}} r(m)r(n).
\]
Since all terms on the right-hand side above are non-negative, we may retain only those with \(k\mid n\), which gives the following crude lower bound for \(S_2\):
\begin{align}
    \label{S2lower}
    S_2 \ge \phi(q)\sum_{k \le q}\frac{(\log k )^\ell}{k^\sigma}r(k)\sum_{\substack{m,n \ge 1 \\ mk \equiv nk \pmod q}} r(m)r(n).
\end{align}
Combining the fact that \(q\) is prime with the argument in \cite[p. 10]{yang2023omega}, we obtain
\[
\phi(q)\sum_{\substack{m,n \ge 1 \\ mk \equiv nk \pmod q}} r(m)r(n) = \phi(q)\sum_{\substack{m,n \ge 1 \\ m \equiv n \pmod q}} r(m)r(n).
\]
Substituting this into \eqref{S2lower} and combining with \eqref{S1eq}, we obtain
\begin{align}
    \label{S2S1lower}
    \frac{S_2}{S_1} \ge \sum_{k \le q}\frac{(\log k )^\ell}{k^\sigma}r(k) =: Q_\ell(q).
\end{align}

\par
If \(k > q\), then \(k^{-\sigma} \le q^{-\sigma/2} k^{-\sigma/2}\), hence
\[
 \sum_{k > q}\frac{(\log k )^\ell}{k^\sigma}r(k) \le q^{-\sigma/2} \sum_{k \in \mbn}\frac{(\log k )^\ell}{k^{\sigma/2}}r(k).
\]
The sum on the right-hand side of the above formula is bounded by
\[
\sum_{k=1}^{\infty}\frac{r(k)}{k^{\sigma/4}}=\prod_{p\leq Y}\Bg(1-\frac{1}{2p^{\sigma/4}}\Bg)^{-1}.
\]
The prime number theorem gives that 
\[
\log \prod_{p\leq Y}\Bg(1-\frac{1}{2p^{\sigma/4}}\Bg)^{-1}\ll_{\sigma}\sum_{p\leq Y}p^{-\frac{\sigma}{4} }\ll_\sigma\frac{Y^{1-\sigma/4}}{\log Y}=o(\log q),
\]
thus, 
\[
\sum_{k \in \mbn}\frac{(\log k )^\ell}{k^{\sigma/2}}r(k) \le q^{o(1)}.
\]
The above argument shows that
\begin{align}
    \label{Qeq}
    Q_\ell(q) = \sum_{k \in \mbn}\frac{(\log k )^\ell}{k^\sigma}r(k) -o(1).
\end{align}
Since $Y\ge 2$ holds for sufficiently large $q$, we have $r(2)=1/2$. Thus, 
\[
\sum_{k\in\mbn}\frac{(\log k)^\ell}{k^\sigma}r(k)
	\geq\sum_{m\in\mbn}r(2m)\frac{(\log 2m)^\ell}{(2m)^\sigma}\geq\frac{(\log 2)^\ell}{2^{\sigma+1}}\sum_{m\in\mbn}\frac{r(m)}{m^\sigma}.
\]
Here, we employ the fact that \(r(n)\) is a completely multiplicative function. Combining this with \eqref{Qeq}, we obtain the following lower bound for \(Q_\ell(q)\):
\begin{align}
    \label{Qlower}
    Q_\ell(q) \gg \sum_{k \in \mbn} \frac{r(k)}{k^\sigma}.
\end{align}

\par
Set 
\[
F(\sigma) := \sum_{k \in \mbn} \frac{r(k)}{k^\sigma}.
\]
Then, by the definition of $r(n)$, we have
\[
F(\sigma)=\prod_{p\leq Y}\Bg(1-\frac{1}{2p^\sigma}\Bg)^{-1}.
\] 
Thus, 
\begin{align*}
    		\log F(\sigma)&=-\sum_{p\leq Y}\log\Bg(1-\frac{1}{2p^\sigma}\Bg)
		\\&=\sum_{p\leq Y}\frac{1}{2p^\sigma}+O\Bg(\sum_{p\leq Y}\frac{1}{p^{2\sigma}}\Bg)
		\\&=\frac 12\sum_{p\leq Y}\frac{1}{p^\sigma}+O_\sigma(1).
\end{align*}
The prime number theorem shows that when \(1/2<\sigma<1\), one has
\[
\sum_{p \le Y} \frac{1}{p^\sigma}= \Bg(\frac{1}{1-\sigma}+o(1)\Bg) \frac{Y^{1-\sigma}}{\log Y}.
\]
Recalling that \(Y= \eta \log q \log\log q\), we get
\begin{align}
    \label{Feq}
    \log F(\sigma) = \Bg(\frac{\eta^{1-\sigma}}{2(1-\sigma)}+o(1)\Bg)  \frac{(\log q)^{1-\sigma}}{(\log\log q)^{\sigma}}.
\end{align}
Combining \eqref{max1}, \eqref{S2S1lower}, \eqref{Qlower} and \eqref{Feq}, and using Lemma \ref{lem1}, we complete the proof of Theorem \ref{thm1}.

	\bibliographystyle{siam}
    \bibliography{reference}
\end{document}